\documentclass{amsart}

\usepackage{setspace}
\usepackage{amssymb}
\usepackage[all]{xy}
\usepackage{enumerate}
\usepackage{amsmath}

\usepackage{amsthm}

\title{On Quasi-isometry and Choice}

\author{Samuel M. Corson}

\newtheorem*{SQHS}{Theorem \ref{SQHS}}
\newtheorem*{Notimplychoice}{Theorem \ref{Notimplychoice}}
\newtheorem*{Bottlenecktheorem}{Theorem \ref{Bottlenecktheorem}}

\newtheorem{theorem}{Theorem}
\newtheorem{corollary}[theorem]{Corollary}
\newtheorem{proposition}[theorem]{Proposition}
\theoremstyle{definition}

\newtheorem{lemma}[theorem]{Lemma}

\newtheorem*{theorem*}{Theorem}

\newcommand{\Lev}{\operatorname{Lev}}

\begin{document}

\begin{abstract}  In this note we prove that the symmetry of the quasi-isometry relation implies the axiom of choice, even when the relation is restricted to geodesic hyperbolic spaces.  We show that this result is sharp by demonstrating that symmetry of quasi-isometry in an even more restrictive setting does not imply the axiom of choice.  The ``Bottleneck Theorem'' of Jason Fox Manning \cite{Ma} also implies choice.
\end{abstract}

\maketitle

\begin{section}{Introduction}

A standard benchmark for the deductive strength of a theorem is that it implies the axiom of choice (see for example \cite{K},\cite{HL}, \cite{Hod}, \cite{B},\cite{How}).  By this we mean that by assuming the Zermelo-Fraenkel axioms of set theory (without the axiom of choice) and the theorem under consideration, the axiom of choice can be deduced (see \cite{TZ} for a listing of these axioms).  Let \textbf{ZF} denote the Zermelo-Fraenkel axioms without the axiom of choice.  The axiom of choice cannot be proved from \textbf{ZF} (see \cite{C}).  Thus if a theorem implies choice then a proof of the theorem requires choice or some stronger assumption in set theory.

We demonstrate the deductive strength of two theorems in metric space theory.  We start with some definitions.  If $(S, d_S)$ and $(T, d_T)$ are metric spaces, a function $f:S \rightarrow T$ is a \textbf{quasi-isometry} if there exists $N\in \omega$ such that $B(f(S), N)  = T$ and for all $x,y\in S$,

\begin{center}  $\frac{1}{N}d_S(x, y) - N \leq d_T(f(x), f(y)) \leq Nd_S(x, y) +N$
\end{center}

\noindent where $B(J, p)$ is the closed neighborhood $\{x\in T: d_T(x, J) \leq p\}$.  This definition differs slightly from the standard one which uses two or three parameters (e.g. Definition 8.14 in \cite{BH}) but our definition is easily seen to be equivalent.  The notion of quasi-isometry is extensively used in geometric group theory (see for example \cite{G1}, \cite{G2}).  In case there exists a quasi-isometry from $S$ to $T$ we say $S$ is quasi-isometric to $T$.  It is easy to see that the quasi-isometry relation is reflexive and transitive (in particular this can be proven without using the axiom of choice).  It is a standard exercise to prove that the quasi-isometry relation is symmetric, and we shall see that the proof unavoidably utilizes the axiom of choice.  That is - the symmetry of the quasi-isometry relation on metric spaces implies the axiom of choice (see Corollary \ref{SQC}).

It is natural to ask whether the symmetry of quasi-isometry in more restrictive settings implies choice.  Recall that a metric space $S$ is \textbf{geodesic} if for any two points $x,y\in S$ there is an isometric embedding $\rho:[0, d_S(x,y)] \rightarrow S$  with $\rho(0) = x$ and $\rho( d_S(x,y)) = y$ (the image of which is called a \textbf{geodesic segment}).  Geodesic segments need not be unique, but a choice of geodesic segment for points $x, y \in S$ will be denoted $[x,y]$.  A geodesic space $S$ is $\delta$-\textbf{hyperbolic} if  for any $x,y,z\in S$  we have $[x,y]\subseteq B([x,z]\cup[y,z], \delta)$ and is \textbf{hyperbolic} if it is $\delta$-hyperbolic for some $\delta$. Although there is a definition for hyperbolicity in a non-geodesic setting (involving the Gromov product), \textit{all hyperbolic spaces in this paper will be geodesic}.  Obviously $\delta_1$-hyperbolicity implies $\delta_0$-hyperbolicity if $\delta_0 \geq\delta_1$. If $S$ is $\delta$-hyperbolic then by scaling the metric by $\lambda>0$ one sees that $S$ is quasi-isometric (via the identity map) to a $\lambda\delta$-hyperbolic space.  A $0$-hyperbolic space is more commonly called an $\mathbb{R}$-tree, and $0$-hyperbolicity implies that geodesics are unique.

Let \textbf{SQHS} (symmetry of quasi-isometry on hyperbolic spaces) denote the assertion that if hyperbolic space $S$ is quasi-isometric to hyperbolic space $T$, then $T$ is quasi-isometric to $S$.  We prove the following:

\begin{theorem}\label{SQHS}  \textbf{SQHS} implies the axiom of choice.
\end{theorem}

From Theorem \ref{SQHS} one immediately obtains the result mentioned earlier:

\begin{corollary}\label{SQC}  The symmetry of quasi-isometry on metric spaces implies the axiom of choice.
\end{corollary}

One can ask whether Theorem \ref{SQHS} can be strengthened by further restricting the symmetry of quasi-isometry to a smaller class of metric spaces.  The most natural choice would be the restriction to the class of $\mathbb{R}$-trees.  It turns out that this restriction is too narrow to imply the axiom of choice.  In particular we have the following, which demonstrates the sharpness of Theorem \ref{SQHS}:

\begin{theorem}\label{Notimplychoice}  The symmetry of the quasi-isometry relation between $\mathbb{R}$-trees follows from \textbf{ZF} alone.
\end{theorem}

In fact we prove the stronger claim that if $f:T \rightarrow S$ is a quasi-isometry with $T$ an $\mathbb{R}$-tree then there exists a quasi-isometry $h: S\rightarrow T$.  This highlights the contrast between $0$-hyperbolic spaces and those which are simply $\delta$-hyperbolic with $\delta>0$ (and therefore quasi-isometric to a $\delta'$-hyperbolic space for any $\delta'>0$ via scaling).

We move on to another result in quasi-isometry.  In \cite{Ma} Jason Fox Manning proved the following (Theorem 4.6), which we call the Bottleneck Theorem:

\begin{theorem*} (J. F. Manning)  Let $Y$ be a geodesic metric space.  The following are equivalent:

\begin{enumerate}\item There exists a simplicial tree $\Gamma$ which is quasi-isometric to $Y$.

\item There is some $\Delta >0$ so that for all $x, y \in Y$ there is a midpoint $m = m(x,y)$ with $d(x, m) = d(y,m) = \frac{1}{2}d(x,y)$ and the property that any path from $x$ to $y$ must pass within less than $\Delta$ of the point $m$.

\end{enumerate}
\end{theorem*}

Part (1) was originally expressed ``$Y$ is quasi-isometric to some simplicial tree.''  However what is exhibited in Manning's proof is a simplicial tree which is quasi-isometric to $Y$, so in light of Theorem \ref{SQHS} we express part (1) as we do.  The final main result of this note is the following:

\begin{theorem}\label{Bottlenecktheorem}  The Bottleneck Theorem  implies the axiom of choice.
\end{theorem}

The proofs of Theorems \ref{SQHS} and \ref{Bottlenecktheorem} are similar and will involve a construction given in  Section \ref{thefirstsectionofthebody}. All of the theorems will then be proved in Section \ref{lastsection}.  We use the following formulation of the axiom of choice:

\textit{If $\mathcal{Z}$ is a nonempty set consisting of pairwise disjoint nonempty sets then there exists a set $A$ such that $A \cap X$ has cardinality one for all $X\in \mathcal{Z}$}.
\end{section}

\begin{section}{The Graph $\Gamma_0$}\label{thefirstsectionofthebody}

The construction and proofs in this section will all be carried out in \textbf{ZF}, without using any other assumptions.  Let $\mathcal{Z}$ be a nonempty collection of pairwise disjoint nonempty sets.  We construct a graph $\Gamma_0 = \Gamma_0(V, E)$ with labelled edges.  For our set of vertices we take $V(\Gamma_0) = ((\bigcup \mathcal{Z})\times (\omega\setminus \{0\}))\cup \{b\}$ where $b\notin \bigcup \mathcal{Z}$.  The graph $\Gamma_0$ will have no edge from a vertex to itself.  Between two distinct vertices there will either be no edge or two edges (with one edge labeled by one vertex, and the other edge labeled by the other vertex).  If there exists an edge between the points $v_0, v_1 \in V$ we let $E(\{v_0, v_1\}, v_0)$ denote the edge between the two vertices which is labeled by $v_0$ and similarly for $E(\{v_0, v_1\}, v_1)$.  For distinct vertices $v_0$ and $v_1$ let $E(\{v_0, v_1\}, v_0)$ and $E(\{v_0, v_1\}, v_1)$ be in $E(\Gamma_0)$ if one of the following holds:

\begin{enumerate}

\item $v_i = b$ and $v_{1-i} = (y, 1)$ for some $y\in \bigcup \mathcal{Z}$

\item $v_i = (x_i, n)$ and $v_{1-i} = (x_{1-i}, m)$ for some $X\in \mathcal{Z}$; $x_i, x_{1-i} \in X$; and $m, n\in \omega \setminus \{0\}$ with $|m-n|\leq 1$

\end{enumerate}

In other words, for each $X\in \mathcal{Z}$ and $n \geq 1$ the induced subgraph on $(X \times \{n\}) \cup (X\times \{n+1\})$ is complete with two edges connecting each pair of distinct vertices, and each vertex in $(\bigcup \mathcal{Z}) \times \{1\}$ shares two edges with the vertex $b$.  We now consider $\Gamma_0$ as a metric graph as follows.  Where there is an edge $E(\{v_0, v_1\}, v_0)$, we attach a compact interval $[-\frac{1}{2},\frac{1}{2}]$ with $-\frac{1}{2}$ being identified with $v_0$ and $\frac{1}{2}$ being identified with $v_1$.  Endow $\Gamma_0$ with the metric given by letting the distance $d_0$ between two points be the minimal length of the path needed to connect them by moving along the attached metric intervals.

It is clear that $\Gamma_0$ is path connected and the path metric defined above makes $\Gamma_0$ a geodesic metric space (with geodesics not being unique in general).  For each point $y\in \Gamma_0$ (not necessarily a vertex) we define the level $\Lev(y)$ by $\Lev(y) = \lfloor d_0(y, b) \rfloor$.  Here $\lfloor \cdot \rfloor$ denotes the floor function.  In particular $\Lev((x,n)) = n$ for any $(x,n)\in V(\Gamma_0)$.  Define the \textbf{base} of $\Gamma_0$ to be the set of all points of level $0$ (i.e. those points of distance $<1$ from $b$).  For each $X\in \mathcal{Z}$ define the $X$-\textbf{arm} to be the set of points of level $\geq 1$ which are distance at most $\frac{1}{2}$ away from a vertex $(x, n) \in V(\Gamma_0)$ with $x\in X\in \mathcal{Z}$ and $n\geq 1$.  Thus each point in $\Gamma_0$ is either in the base or is in an $X$-arm for exactly one $X\in \mathcal{Z}$ (uniqueness of the arm follows from the pairwise disjointness of the elements of $\mathcal{Z}$).

We prove some lemmas which will aid us in proving the main proposition of this section.

\begin{lemma} \label{diameter} The base, as well as each level of each arm, is of diameter $\leq 2$.
\end{lemma}

\begin{proof}  Notice that each element of the base is distance $< 1$ away from the point $b$, whence the bound on the diameter of the base follows.  Each point in $\Gamma_0$ is at most distance $\frac{1}{2}$ away from some vertex.  If elements $y_0, y_1$ are in the $X$-arm and are of level $n$, then in particular each $y_i$ is distance $\leq \frac{1}{2}$ away from a vertex $v_i$ with second coordinate $n$ or $n+1$.  Then there is a path from $y_0$ to $v_0$ of length $\leq \frac{1}{2}$, a path from $v_0$ to $v_1$ of length $0$ or $1$, and a path from $v_1$ to $y_1$ of length $\leq\frac{1}{2}$.
\end{proof}

\begin{lemma}\label{geodesicmotion}  If $\rho:[0, r] \rightarrow \Gamma_0$ is a geodesic let $\rho^{-1}(V(\Gamma_0))= \{a_0, \ldots, a_k\}$ with $a_{i+1} = a_i +1$.  Then $|a_0-0|, |r-a_k| < 1$.  If $\Lev(\rho(a_i))$ is constant then $k\leq 1$.  If $k \geq 2$ then the sequence $\Lev(\rho(a_i))$ is either increasing, or decreasing, or decreasing to $0$ and then increasing (by $1$ unit in each case).
\end{lemma}

\begin{proof}  That  $|a_0-0|, |r-a_k| < 1$ holds is clear.  Certainly $|\Lev(\rho(a_i)) - \Lev(\rho(a_{i+1}))| \leq 1$.  If $\Lev(\rho(a_i)) = \Lev(\rho(a_{i+1}))$ then $\rho$ may not pass through any other vertices since in particular there would be $a_j$ with $j\neq i, i+1$ and $|\Lev(\rho(a_j)) - \Lev(\rho(a_i))| \leq 1$, so that $d_0(\rho(a_j), \rho(a_i)), d_0(\rho(a_j), \rho(a_{i+1})) \leq 1$.  Thus in this case  $i=0$ and $i+1 = k = 1$.  Notice that $\rho(a_i)$ and $\rho(a_j)$ cannot both be in the same level on the same arm, or both be in the base,  if  $k \geq 2$ and $i \neq j$  (by a similar argument).  Supposing that $k \geq 2$ we therefore know that either the sequence $\Lev(\rho(a_i))$ increases by $1$ at each index; or decreases by $1$ at each index; or decreases to $0$ by $1$ at each index, after which $\rho$ moves up another arm and $\Lev(\rho(a_i))$ increases by $1$ at each index.
\end{proof}

\begin{lemma} \label{tightquarters} Suppose $w\in [x,y]\subseteq \Gamma_0$.  Then any path from $x$ to $y$ must pass within distance $2$ of $w$. 
\end{lemma}

\begin{proof}  The claim is trivial if either $d(x, w)\leq 2$ or $d(y, w)\leq 2$.  Otherwise, let $\rho:[0, d(x,y)] \rightarrow \Gamma_0$ be the geodesic associated with the segment $[x,y]$ with $ \rho(0) = x$.  Let $a_0, a_1, \ldots, a_k$ be the sequence described in Lemma \ref{geodesicmotion} and let $0\leq j\leq k$ be such that $w\in [\rho(a_j), \rho(a_{j+1})] \setminus \{\rho(a_{j+1})\} \subseteq [x,y]$.  Since $d_0(w, x), d_0(w, y) >2$ we know that $1\leq j \leq k-2$.  By Lemma \ref{geodesicmotion} it is either the case that $\Lev(\rho(a_j))$ is increasing, or decreasing, or decreasing to $0$ and then increasing, by increments of $1$.  In case $\Lev(\rho(a_i))$ is increasing, we know that $x$ is either in the base or in the same arm as $w$ and $y$, with $\Lev(y) > \Lev(w)>\Lev(x)$, and removing $B(w, 2)$ from the arm removes all vertices from that arm of level $\Lev(\rho(a_j))$ (since all such vertices are distance $1$ from each other and $d_0(\rho(a_j), w) <1$).  There is no combinatorial path in $\Gamma_0$ from $\rho(a_0)$ to $\rho(a_k)$ which does not pass through a vertex of level $\Lev(\rho(a_j))$, so no path from $x$ to $y$ avoids $B(w, 2)$.  Similar proofs for $\Lev(\rho(a_i))$ decreasing and $\Lev(\rho(a_i))$ decreasing to $0$ and then increasing prove the claim.
\end{proof}

\begin{lemma}\label{hyp}  The geodesic space $\Gamma_0$ is $2$-hyperbolic.
\end{lemma}

\begin{proof}  Let $x,y, z\in \Gamma_0$ be given.  Notice that $[x,z]\cup [z,y]$ gives a path from $x$ to $y$, and Lemma \ref{tightquarters} shows that $[x, y] \subseteq B([x,z]\cup[y,z], 2)$.
\end{proof}

\begin{lemma}\label{midpoint}  Given $x,y\in \Gamma_0$ there exists a point $m$ with $d_0(m, x) = d_0(m,y) = \frac{1}{2}d_0(x,y)$ such that any path from $x$ to $y$ must pass within less than $\Delta = 3$ of the point $m$.
\end{lemma}

\begin{proof}  Let $[x,y]$ be any geodesic, let $m\in [x,y]$ be the point such that $d_0(m, x) = d_0(m,y) = \frac{1}{2}d_0(x,y)$.   Applying Lemma \ref{tightquarters} with $m = w$ we know that any path from $x$ to $y$ must pass within distance $2$ of $m$, so any path from $x$ to $y$ must pass within distance less than $3$ of $m$.
\end{proof}

\begin{lemma}\label{closeenough}  If $x, y\in \Gamma_0$ and $L >0$ then there exists $z\in V(\Gamma_0)$ such that $d_0(x,z) >L$, $d_0(y, z)>L$, $\Lev(z) >L$, and any path from $x$ to $z$ comes within distance $4$ of $y$.
\end{lemma}

\begin{proof}  The claim is straightforward to prove if $d_0(x, y) \leq 4$, so assume $d_0(x, y) >4$.  Then $x$ and $y$ cannot both be in the base or both be on the same level of the same arm by Lemma \ref{diameter}.  Let $L_0\in \omega$ with $L_0>L + \Lev(y) + \Lev(x)+2$.  If $x, y$ lie on the same arm with $\Lev(x) >\Lev(y)$, or if $y$ is in the base, then let $z\in V(\Gamma_0)$ be on a different arm from $x$ with $\Lev(z) \geq L_0$.  Then certainly $d_0(x, z)> d(y, z) \geq \Lev(z)-1 >L$.  Letting $v\in V(\Gamma_0)$ be the unique point on the same arm as $y$ (or $v=b$ in case $y$ is in the base) satisfying $\Lev(v) = \Lev(y)$ and $v\in [x,z]$, we have by Lemma \ref{tightquarters} that any path from $x$ to $z$ must pass within distance $2$ of $v$, and since $d_0(v, y)\leq 2$ we know any path from $x$ to $z$ must pass within distance $4$ of $y$.

Else, select $z\in V(\Gamma_0)$ with $\Lev(z) \geq L_0$ and with $z$ on the same arm as $y$.  Then $\Lev(z) >L$ and $d_0(x, z) \geq d_0(b, z) - d_0(x, b) \geq \Lev(z) - (\Lev(x) +2) >L$, and $d_0(y, z) > L$ follows as well.  As before, $y$ is within distance $2$ of any geodesic segment $[x, z]$ and we argue similarly to obtain the same conclusion.
\end{proof}

The proofs of Theorems \ref{SQHS} and \ref{Bottlenecktheorem} hinge on the following property of $\Gamma_0$:

\begin{proposition}\label{thebigproposition}  If $g: \Gamma \rightarrow \Gamma_0$ is a quasi-isometry from a simplicial tree $\Gamma$ then there exists a set $A$ such that $A \cap X$ has cardinality $1$ for all $X\in \mathcal{Z}$.
\end{proposition}

We spend the remainder of this section proving Proposition \ref{thebigproposition}.  Let $g:\Gamma \rightarrow \Gamma_0$ be a quasi-isometry, with $\Gamma$ a simplicial tree with simplicial metric $d$.  Let $N \geq 4$ be an associated quasi-isometry constant for $g$.  We describe a pruning process for the tree $\Gamma$.  Call a vertex $w\in V(\Gamma)$ \textbf{terminal} if $w$ is of valence one.  Let $\Gamma^{(1)} = \Gamma'$ be the tree $\Gamma$ with all terminal vertices and adjoining edges removed, and in general let $\Gamma^{(n+1)} = (\Gamma^{(n)})'$.  It is clear that for all $n$ the graph $\Gamma^{(n)}$ is also a simplicial tree.  Also, for $n\geq 2$ and $w\in V(\Gamma^{(n)})\setminus V(\Gamma^{(n-1)})$ there exists some $w'\in V(\Gamma^{(n-1)}) \setminus V(\Gamma^{(n-2)})$ which is adjacent to $w$.  One can show by induction on $n$ that given adjacent vertices $w, w'$ such that $w\in V(\Gamma^{(n)})\setminus V(\Gamma^{(n-1)})$ and $w'\in V(\Gamma^{(n-1)}) \setminus V(\Gamma^{(n-2)})$, there is no geodesic segment $\gamma$ starting at $w$ of length greater than $n$  with $w'\in \gamma$.  We claim that the pruning process on $\Gamma$ stabilizes.  More concretely, letting $K = 7N^2$ we have the following:

\begin{lemma}  $\Gamma^{(K)} = \Gamma^{(K+1)}$
\end{lemma}

\begin{proof}  Suppose for contradiction that $w\in V(\Gamma^{(K)}) \setminus V(\Gamma^{(K+1)})$.  Let $w_0 = w$ and pick $w_1\in V(\Gamma^{(K-1)})\setminus V(\Gamma^{(K)})$ which is adjacent to $w$.  Continue in this manner so that $w_n\in V(\Gamma^{(K-n)}) \setminus V(\Gamma^{(K-n+1)})$ is adjacent to $w_{n-1}$ for $n \leq K$.  Then $w_K\in V(\Gamma) \setminus V(\Gamma^{(1)})$ is distance $K$ away from $w$.  

Suppose that $w' \in \Gamma$ is distance at least $2K$ from $w_K$ and suppose for contradiction that $w \notin [w_K, w']$.  Let $n\in \omega$ be least such that $w_n\in  [w_K, w']$, and by assumption $n \geq 1$.  Notice that $d(w_n, w') \geq 2K-(K-n) = K+n$.  But the geodesic segment $[w, w'] = [w, w_n]\cup [w_n, w']$ gives a geodesic beginning at $w$, passing through $w_1$ which is of length $\geq n+(K+n)>K$, a contradiction.  Thus $d(w', w_K) \geq 2K$ implies $w \in [w_K, w']$.

Notice that $d_0(g(w), g(w_K)) \geq \frac{1}{N}d(w, w_K) - N = \frac{K}{N} - N = 6N$.  Letting $x = g(w)$ and $y = g(w_K)$ we pick a $z\in \Gamma_0$ as in Lemma \ref{closeenough} with $L = (2K+1)N + N^2$.  Select $w'\in \Gamma$ such that $d_0(g(w'), z) \leq N$.  Now

\begin{center}

 $d(w', w_K) \geq \frac{1}{N}d_0(g(w'), g(w_K)) -1$

$> \frac{1}{N}(d_0(z, g(w_K)) -N) -N$

$\geq \frac{1}{N}(L-N) -N$

$\geq (2K +N) -N=2K$
\end{center}

Letting $w = v_0, v_1, \ldots , v_p$ be the vertices in $\Gamma$ on the geodesic from $w$ to $w'$, listed in increasing distance from $w$, we know that $d(v_i, w_K) \geq K$ for all $i$.  Then $d_0(g(v_i), g(w_K)) \geq 6N$.  Also, $d_0(g(v_i), g(v_{i+1})) \leq 2N$ for all $i$ and $d_0(g(v_p), g(w')) \leq 2N$.  Notice that $$[g(v_0), g(v_1)] \cup [g(v_1), g(v_2)]\cup\cdots \cup [g(v_{p-1}), g(v_p)]\cup [g(v_p), g(w')]\cup [g(w'), z]$$ gives a path from $g(w)$ to $z$ which must come within distance $4$ of $g(w_K)$.  But this requires for some $i$ the inequality $d_0(g(v_i), g(w_K)) \leq 4 + 4N$, a contradiction.
\end{proof}

It is clear that the subtree $\Gamma^{(K)}$ satisfies $B(\Gamma^{(K)}, K) = \Gamma$, and so the inclusion map $\Gamma^{(K)} \rightarrow \Gamma$ is a quasi-isometry, and composing $g$ with inclusion gives a quasi-isometry from $\Gamma^{(K)}$ to $\Gamma_0$ (that the composition of quasi-isometries is a quasi-isometry is easily provable in \textbf{ZF}).  Thus we may assume without loss of generality that $\Gamma$ has only vertices of valence $2$ or greater.  Let $K = 7N^2$ as before.  Fix once and for all a vertex $v\in V(\Gamma)$ such that $d_0(g(v), b) \leq 3N$.  Such a selection of $v$ is possible by choosing $w\in \Gamma$ with $d_0(g(w), b) \leq N$ and a vertex $v\in V(\Gamma)$ with $d(v, w) \leq \frac{1}{2}$.

\begin{lemma}\label{valence2}  If $v'\in V(\Gamma)$ with $d(v, v') \geq K$ then $v'$ has valence exactly $2$.
\end{lemma}

\begin{proof}  Suppose for contradiction that $v'$ has valence at least $3$ and let $K_0 = d(v', v) \geq K$.  Select $w_{1, 0}, w_{2, 0}, \ldots , w_{K, 0}= w_0 \in V(\Gamma)$  such that $d(w_0, v) = d(w_0, v') + d(v', v)$ and $v', w_{1, 0}, w_{2, 0}, \ldots , w_{K, 0}= w_0$ are the vertices through which the geodesic from $v'$ to $w_0$ passes, listed in order.  Thus $d(v', w_0) = K$.  Let $K_1= N^2(K+K_0 +9)$.  As $v'$ is of valence at least $3$ we also have vertices $w_{1, 1}, w_{2,1}, \ldots , w_{K_1,1}=w_1$  where $d(v, w_1) = d(v, v') +d(v', w_1)$;  $v', w_{1, 1}, \ldots, w_{K_1,1}$ are the vertices through which the geodesic from $v'$ to $w_1$ move; and $w_{1, 1} \neq w_{1, 0}$.  These selections are possible because $v'$ is of valence $\geq 3$ and since all $w_{i, j}$ are of valence at least $2$.  Making such a finite number of choices, where the number of choices is already known, is still within the purview of \textbf{ZF}.  Let $v=v_0, v_1, v_2, \ldots, v_{K_0} = v'$ be the vertices of the geodesic segment $[v, v']$ listed in the order in which they are traversed on the geodesic from $v$ to $v'$.

Notice

\begin{center}  $d_0(g(v'), b) \geq d_0(g(v'), g(v)) - d_0(g(v), b)$

$\geq (\frac{1}{N}d(v', v) -N) -3N$

$\geq \frac{K}{N} - N -3N$

$=3N$

\end{center}

\noindent and by the same token we have $d_0(g(w_{i, j}), b) \geq 3N$ for all $w_{i, j}$ which we have defined.  Then all $g(w_{i, j})$ and also $g(v')$ are not in the base of $\Gamma_0$, and we claim in fact that they must be in the same arm, say the $X$-arm.  We prove this for $g(w_0)$ and $g(w_1)$ and the proof in all other cases is completely analogous.  If $g(w_0)$ and $g(w_1)$ are in different arms then any path from $g(w_0)$ to $g(w_1)$ must pass through $b$.  Then the path given by 

\begin{center}

$[g(w_0) = g(w_{K, 0}), g(w_{K-1, 0})]\cup\cdots \cup [g(w_{2, 0}), g(w_{1, 0})]\cup[g(w_{1, 0}), g(v')]$

$\cup [g(v'), g(w_{1, 1})]\cup \cdots \cup [g(w_{K_1-1, 1}), g(w_{K_1, 1}) = g(w_1)]$

\end{center}

\noindent must pass through $b$, but each listed segment of the path is of length $\leq 2N$ and so this would bring either $g(v')$ or some $g(w_{i, j})$ within distance $N$ of $b$, which is a contradiction.  We further note that

\begin{center}$d_0(g(w_0), b) \geq d_0(g(w_0), g(v)) - d_0(g(v), b)$

$\geq (\frac{2K}{N} - N) - 3N$

$= 10N$
\end{center}

\noindent and

\begin{center}  $d_0(g(w_0), b) \leq d_0(g(w_0), g(v)) + d_0(g(v), b)$

$\leq (Nd(w_0, v) + N) +3N$

$= N(K +K_0) +4N$
\end{center}

\noindent and

\begin{center} $d_0(g(w_1), b) \geq d_0(g(w_1), g(v)) - d_0(g(v), b)$

$\geq (\frac{1}{N}d(w_1, v) -N) -3N$

$>(\frac{1}{N}d(w_1, v') -N) -3N$

$\geq \frac{N^2(K+K_0 +9)}{N} - 4N$

$\geq N(K+K_0) +5N$
\end{center}

The point $g(w_0)$ must come within distance $2$ of any geodesic $[g(v), g(w_1)]$ since $g(w_0)$ and $g(w_1)$  lie on the $X$-arm  with $10 N\leq \Lev(g(w_0))\leq \Lev(g(w_1)) - N$ and $\Lev(g(v)) \leq 3N$ (the point $g(v)$ needn't lie on the $X$-arm).  Then by Lemma \ref{tightquarters} we know that the path given by 

\begin{center}
$ [g(v) = g(v_0), g(v_1)]\cup[g(v_1), g(v_2)]\cup\cdots \cup [g(v_{K_0-1}), g(v_{K_0}) = g(v')]\cup[g(v'), g(w_{1, 1)}]\cup$

$\cdots \cup [g(w_{K_1-1, 1}), g(w_{K_1, 1})= g(w_1)]$ 

\end{center}

\noindent must pass within distance $4$ of $g(w_0)$.  But then some $g(v_i)$ or $g(w_{j,1})$ must be within $N +4$ of $g(w_0)$, contradicting $d_0(g(w_0), g(w_{j, 1})), d_0(g(w_0), g(v_i)) \geq \frac{K}{N} - N = 6N$ for all $0\leq i \leq K_0$ and $1 \leq j \leq K_1$.
\end{proof}

Let $W$ be the set of all vertices of distance exactly $2K$ from $v$.

\begin{lemma}  For each $w\in W$ the point $g(w)$ is in an arm of $\Gamma_0$.  Also, for each $X\in \mathcal{Z}$ there is a unique $w\in W$ with $g(w)$ in the $X$-arm.
\end{lemma}

\begin{proof}  If $w\in W$ then 

\begin{center}
$d_0(g(w), b) \geq d_0(g(w), g(v)) - d_0(g(v), b)$

$\geq (\frac{2K}{N} -N) -3N = 10N$
\end{center}

so that $\Lev (g(w)) \geq 10N >0$ and $g(w)$ is in an arm.

Suppose that for distinct $w_0, w_1 \in W$ we have $g(w_0)$ and $g(w_1)$ in the $X$-arm.  As $w_0$ and $w_1$ are distinct we know that $d(w_0, w_1) \geq 2K$ since letting $v'\in V(\Gamma)$ satisfy $[v, w_0]\cap [v, w_1] = [v, v']$, we have $d(v, v') <K$ by Lemma \ref{valence2}.  Then $d_0(g(w_0), g(w_1)) \geq \frac{2K}{N} -N = 10N$.  Let without loss of generality $\Lev(g(w_1)) \geq \Lev(g(w_0)) + 10N -4$.  Let $v = v_0, v_1, \ldots, v_{2K} = w_1$ be the vertices in $[v, w_1]$ listed in increasing distance from $v$.  As in the proof of Lemma \ref{valence2} the path given by $[g(v_0), g(v_1)]\cup\cdots \cup [g(v_{2K-1}), g(v_{2K})]$ is distance at most $4$ away from $g(w_0)$, so some $g(v_i)$ must be within $N + 4$ of $g(w_0)$, contradicting $d_0(g(v_i), g(w_0)) \geq \frac{K}{N} -N = 6N$.

Let $X\in \mathcal{Z}$ be given and select $x\in \Gamma$ such that $g(x)$ is in the $X$-arm and $\Lev(g(x)) \geq 2KN +2N^2 +3N$.  Then 

\begin{center}  $d(x, v) \geq \frac{1}{N}d_0(g(x), g(v)) -1$

$> \frac{1}{N}(d_0(g(x), b)-d_0(g(v), b)) -N$

$\geq \frac{1}{N}(d_0(g(x), b)-3N) -N$

$\geq \frac{2KN +2N^2}{N}-N = 2K +N$
\end{center}

Let $w\in [v, x]$ be the vertex such that $d(w, v) = 2K$ and since all vertices in the geodesic segment $[w,x]$ are at least distance $2K$ from $v$ we may argue as before that $g(w)$ is also in the $X$-arm.
\end{proof}

To finish the proof of Proposition \ref{thebigproposition} we let $h:W \rightarrow \bigcup \mathcal{Z}$ be defined by

\begin{center}

$h(w) = \begin{cases}\pi_1(g(w))$ if $g(w)\in V(\Gamma_0)\\\pi_1(\pi_2(E))$ where $g(w)\notin V(\Gamma_0)$ lies on the edge $E  \end{cases}$

\end{center}

where $\pi_1$ and $\pi_2$ are the projections to the first and second coordinates, respectively.  Thus $h(w)$ gives the first coordinate of $g(w)$ provided $g(w)\in V(\Gamma_0)$ and otherwise $h$ gives the first coordinate of the label of the edge on which $g(w)$ lies.  Let $A = h(W)$.  The check that $A$ satisfies the conclusion of Proposition \ref{thebigproposition} is straightforward.

\end{section}

\begin{section}{The Proofs of the Main Results}\label{lastsection}

In this section we restate and prove Theorems \ref{SQHS}, \ref{Bottlenecktheorem}, and then \ref{Notimplychoice}.

\begin{SQHS}  \textbf{SQHS} implies the axiom of choice.
\end{SQHS}

\begin{proof} Assume \textbf{SQHS}, let $\mathcal{Z}$ be a nonempty collection of disjoint nonempty sets, and let $\Gamma_0$ be as constructed in Section \ref{thefirstsectionofthebody}.  We give a simplicial tree $\Gamma_1$ as follows:  Let $V(\Gamma_1) = (\mathcal{Z} \times (\omega\setminus\{0\}))\cup \{B\}$ where $B\notin \mathcal{Z}$.  Let $(v_0, v_1)\in E(\Gamma_1)$ if one of the following holds:

\begin{enumerate}

\item $v_i = B$ and $v_{1-i} = (Y, 1)$ for some $Y\in \mathcal{Z}$

\item $v_i = (Y, n)$ and $v_{1-i} = (Y, m)$ for some $Y\in \mathcal{Z}$ and $m = n\pm 1$

\end{enumerate}

Let $f: \Gamma_0 \rightarrow \Gamma_1$ be given in the following way.  Map vertices to vertices by letting $f(b) = B$ and $f((x, n)) = (X, n)$  where $x\in X$.  Map all points on an edge between $(x_0, n)$ and $(x_1, n)$ to $(X, n)$ (where $x_0, x_1\in X$) and map points on an edge between vertices of differing levels so that the restriction of $f$ to this edge is an isometry.  Notice that $f$ is onto and

\begin{center}  $d_0(x, y) - 2 \leq d_1(f(x), f(y)) \leq d_0(x, y)$

\end{center}

\noindent where $d_1$ is the simplicial metric on $\Gamma_1$.  Since $\Gamma_1$ is $0$-hyperbolic and $\Gamma_0$ is $2$-hyperbolic (Lemma \ref{hyp}) we have by \textbf{SQHS} a quasi-isometry $g: \Gamma_1 \rightarrow \Gamma_0$, which implies a selection for the collection $\mathcal{Z}$ by Proposition \ref{thebigproposition}.
\end{proof}

\begin{Bottlenecktheorem}  The Bottleneck Theorem  implies the axiom of choice.
\end{Bottlenecktheorem}

\begin{proof}  Assume the Bottleneck Theorem, let $\mathcal{Z}$ be a nonempty collection of nonempty pairwise disjoint sets and let $\Gamma_0$ be as constructed in Section \ref{thefirstsectionofthebody}.  By Lemma \ref{midpoint} we know $\Gamma_0$ satisfies condition (2) of the Bottleneck Theorem, so there exists a simplicial tree $\Gamma_2$ and quasi-isometry $g:\Gamma_2 \rightarrow \Gamma_0$.  By Proposition \ref{thebigproposition} there is a selection on $\mathcal{Z}$ and we have proven the axiom of choice.
\end{proof}

\begin{Notimplychoice}  The symmetry of the quasi-isometry relation between $\mathbb{R}$-trees follows from \textbf{ZF}.
\end{Notimplychoice}

\begin{proof}  We suppose that $f: (T, d) \rightarrow (T', d')$ is a quasi-isometry, with $T$ an $\mathbb{R}$-tree.  Let $N \in \omega \setminus \{0\}$ be a corresponding quasi-isometry constant.  We may assume that both $T$ and $T'$ are nonempty.  Fix $z\in T$.  We define a map $h: (T', d') \rightarrow (T, d)$ by letting $h(x) = y$ where $[z, y] = \bigcap_{y_0\in f^{-1}(B(x, N))}[z, y_0]$.

We check (using only \textbf{ZF}) that the map $h$ is well defined.  We show that for every $x\in T'$ there is a unique $y\in T$ such that $[z,y] = \bigcap_{y_0\in f^{-1}(B(\{x\}, N))}[z, y_0]$, and the axiom of replacement implies that such a map $h$ exists.  Let $x\in T'$ be given.  We know by assumption that $B(x, N)$ intersects the image of $f$ nontrivially, so $f^{-1}(B(x, N))$ is nonempty.  Select $y_0'\in f^{-1}(B(x, N))$.  Notice that if $w\in \bigcap_{y_0\in f^{-1}(B(x, N))}[z, y_0]$ and $w'\in [z,w]$ then $w'\in  \bigcap_{y_0\in f^{-1}(B(x, N))}[z, y_0]$ as well.  Then $\bigcap_{y_0\in f^{-1}(B(x, N))}[z, y_0]$ is a subset of $[z, y_0']$ which is closed under taking elements of geodesic segments from $z$.  Moreover, $\bigcap_{y_0\in f^{-1}(B(x, N))}[z, y_0]$ is compact as an itersection of closed subsets of the compact metric space $[0,y_0']$.  The set $\{t\in[0, d(z, y_0)]: (\exists w\in \bigcap_{y_0\in f^{-1}(B(x, N))}[z, y_0]) d(z, w) = t\}$ is isometric with the compact space $\bigcap_{y_0\in f^{-1}(B(x, N))}[z, y_0]$ and so $\{t\in[0, d(z, y_0)]: (\exists w\in \bigcap_{y_0\in f^{-1}(B(x, N))}[z, y_0]) d(z, w) = t\}$ contains a supremum $s$, which corresponds under the isometry to a point $y\in  \bigcap_{y_0\in f^{-1}(B(x, N))}[z, y_0]$.  Then we know $[z, y]\subseteq  \bigcap_{y_0\in f^{-1}(B(x, N))}[z, y_0]$, and that $\bigcap_{y_0\in f^{-1}(B(x, N))}[z, y_0] \subseteq [z,y] \subseteq [z,y_0']$ is clear.  That $y$ is unique follows from the fact that any $y'\in T$ is uniquely determined by the geodesic $[z,y']$.

We now show (using only \textbf{ZF}) that $h$ is a quasi-isometry and we will be done.  We claim that $M = 9N^2$ is a quasi-isometry constant for $h$.  Note first that for a given $x\in T'$ that $y_0, y_1\in f^{-1}(B(x, N))$ implies that $d'(f(y_0), f(y_1)) \leq 2N$, so that $$\frac{1}{N}d(y_0, y_1) -N \leq d'(f(y_0), f(y_1)) \leq 2N$$ from which we obtain $d(y_0, y_1) \leq 3N^2$.  As $T$ is a tree it follows that $h(x)$ is distance at most $3N^2$ from any $y_0\in  f^{-1}(B(x, N))$, for if $y_0, y_1\in  f^{-1}(B(x, N))$, the equality $[z, y_0]\cap [z, y_1] = [z, w]$ implies $d(w, y_0) \leq 3N^2$.  Then $d(w, h(f(w))) \leq 3N^2 < 9N^2$, so that $B(h(T'),  9N^2) = T$.  Let $x_0, x_1 \in T'$ and select $y_0, y_1\in T$ such that $d'(f(y_0), x_0), d'(f(y_1), x_1)\leq N$.  Then as noted we have $d(y_0, h(x_0)), d(y_1, h(x_1)) \leq 3N^2$.  Then

\begin{center}

$\frac{1}{M}d'(x_0, x_1) - M  \leq \frac{1}{N} d'(x_0, x_1) - 9N^2$

$\leq \frac{1}{N}(d'(x_0, x_1)-N-N)-7N^2$

$\leq \frac{1}{N}(d'(x_0, x_1)-d'(x_0, f(y_0))-d'(x_1, f(y_1))) -7N^2$

$\leq \frac{1}{N}d'(f(y_0), f(y_1))-N-6N^2$

$\leq d(y_0, y_1) - 6N^2$

$\leq d(y_0, y_1) - d(h(x_0), y_0)-d(h(x_1), y_1)$

$\leq d(h(x_0), h(x_1))$

$\leq d(h(x_0), y_0) + d(h(x_1), y_1) + d(y_0, y_1)$

$\leq 6N^2 + d(y_0, y_1)$

$\leq 6N^2 + Nd'(f(y_0), f(y_1)) + N^2$

$\leq 6N^2 +N(d'(x_0, x_1) + d'(x_0, f(y_0)) + d'(x_1, f(y_1))) +N^2$

$\leq 9N^2 + Nd'(x_0, x_1) \leq Md'(x_0, x_1) + M$
\end{center}

\noindent which completes the proof.
\end{proof}

\end{section}

\author{Samuel M. Corson}
\address{Mathematics Department\\
1326 Stevenson Center\\
Vanderbilt University\\
Nashville, TN 37240\\
USA}
\email{samuel.m.corson@vanderbilt.edu}

\end{document}